\documentclass[12pt,a4paper]{article}

\usepackage{latexsym}
\usepackage{amssymb}
\usepackage{amsfonts}
\usepackage{amsthm}
\usepackage{amsmath}
\usepackage[pdftex]{graphicx}
\usepackage{comment}
\usepackage[usenames,dvipsnames]{xcolor}
\usepackage{framed}
\usepackage{enumitem}
\usepackage{booktabs}
\usepackage{caption}

\makeatletter
\g@addto@macro\@floatboxreset\centering
\makeatother

\newcommand{\eps}{{\varepsilon}}
\newcommand{\R}{\mathbb{R}}

\newcommand{\N}{\mathbb{N}}
\renewcommand{\P}{\mathbb{P}}
\newcommand{\E}{\mathbb{E}}

\renewcommand{\d}{\, \mathrm{d}}

\DeclareMathOperator{\var}{Var}

\newcommand{\be}[1]{\begin{equation}\label{#1}}
\newcommand{\ee}{\end{equation}}

\newtheorem{theorem}{Theorem}[section]

\newtheorem{lemma}[theorem]{Lemma}

\newtheorem{remark}[theorem]{Remark}

\begin{document}
\title{Optimally stopping at a given distance from the ultimate supremum of a spectrally negative L\'evy process}
\author{M\'onica B. Carvajal Pinto\footnote{School of Mathematics, University of Manchester. Oxford Road, {\sc Manchester, M13 9PL, United Kingdom.} E-mail: monica.carvajalpinto@manchester.ac.uk/kees.vanschaik@manchester.ac.uk} \quad and \quad Kees van Schaik\footnotemark[1]}
\date{\today}
\maketitle

\begin{abstract} 
\noindent We consider the optimal prediction problem of stopping a spectrally negative L\'evy process as close as possible to a given distance $b \geq 0$ from its ultimate supremum, under a squared error penalty function. Under some mild conditions, the solution is fully and explicitly characterised in terms of scale functions. We find that the solution has an interesting non-trivial structure: if $b$ is larger than a certain threshold then it is optimal to stop as soon as the difference between the running supremum and the position of the process exceeds a certain level (less than $b$), while if $b$ is smaller than this threshold then it is optimal to stop immediately (independent of the running supremum and position of the process). We also present some examples.
\end{abstract}

\medskip 


\medskip


\vspace{0.5cm}

\section{Introduction}\label{sec_intro}

Let $X=(X_t)_{t \geq 0}$ be a spectrally negative L\'evy process starting from $0$ defined on a filtered probability space $(\Omega,\mathcal{F},\mathbf{F},\mathbb{P})$, where $\mathbf{F}=(\mathcal{F}_t)_{t \geq 0}$ is the filtration generated by $X$ which is naturally enlarged (cf. Definition 1.3.38 in \cite{Bichteler02}). We denote its running supremum by
\[ \overline{X}_t := \sup_{s \leq t} X_s \quad \text{for all $t \in [0,\infty]$}. \]
In this paper we consider the optimal prediction problem
\be{eq_main} 
\inf_{\tau} \E \left[ \varphi \left( \overline{X}_\infty-X_\tau \right) \right], 
\ee
where $\varphi: \R_{\geq 0} \to \R$ is a non-negative continuous penalty function and the infimium is taken over all $\mathcal{F}$ adapted stopping times $\tau$. To avoid trivialities we restrict ourselves to the case that $\overline{X}_\infty<\infty$ a.s. and $\E[\varphi(\overline{X}_\infty)]<\infty$. In particular we focus on the case that 
\be{def_pen}
\varphi(x)=(x-b)^2 \quad \text{for some $b>0$}. 
\ee
This means that we are looking for the stopping time $\tau$ such that $X_\tau$ is closest to $\overline{X}_\infty-b$ under a squared error penalty (as commonly used in many areas of optimisation and estimation). 

This is a challenging problem. One reason is that the decision to stop has to be made before the value of the ultimate supremum $\overline{X}_\infty$ is fully known (it is still tractable though due to the homogeneity properties of a L\'evy process). A further complication is that the presence of (unpredictable) negative jumps in $X$ means that a path of $X$ may suddenly jump from high levels to much smaller levels and then even drift off to $-\infty$ before ever returning to anywhere near the previously attained high levels. Our main result fully characterises the solution (under some mild conditions) and shows that there are two types of solutions. If $b$ is smaller than a particular threshold then it is optimal to stop immediately, while if $b$ is larger than this threshold then it is optimal to stop as soon as the difference $\overline{X}_t-X_t$ exceeds a certain level (typically strictly smaller than $b$).

Applications of this problem can for instance be found in (mathematical) finance and insurance. In finance L\'evy processes have received a lot of attention in recent years as an alternative model for the evolution of (the log of) a financial index, extending the classic Black \& Scholes model. See e.g. the textbooks \cite{Schoutens03} \& \cite{Tankov03} and the references therein to name only some. In such a model, the ultimate supremum represents the maximal value the index will attain and hence the problem studied in this paper can be used by an agent to evaluate when is a good (or even the optimal) time to sell shares held in the index. By considering the solution to the problem for different values of $b$, a sequence of `alarms' leading up to the optimal time can be created to inform the agent's investment strategy.

In insurance, a spectrally negative L\'evy process is a popular extension of the classic Cram\'er-Lundberg model modelling the evolution of the surplus associated with a portfolio of products (in the presence of premium income and outgo due to claim payments). See e.g. the textbooks \cite{Asmussen10} \& \cite{Kyprianou13} and the references therein. In this context, the fact that the ultimate supremum is finite means that the so-called Net Profit Condition is not satisfied i.e. on average the portfolio generates a loss per time unit. This may be because the insurer is trying to gain exposure in the consumer market or because a speculative agent is holding the portfolio. The sequence of `alarms' mentioned above can in this application similarly be used by the speculative agent to sell it or for the insurer to start dismantling the portfolio.

An optimisation problem of the type \eqref{eq_main}, where the payoff at any time $t$ i.e. $\varphi (\overline{X}_\infty-X_t)$ is not adapted to $\mathcal{F}_t$ is typically referred to as an optimal prediction problem (and this aspect is exactly what distinguishes optimal prediction problems from the classic field of optimal stopping problems). Such problems have received a lot of attention in recent years, see e.g. \cite{Allaart10, Baurdoux14, Baurdoux16, Bernyk11, Cohen10, DuToit07, DuToit08, DuToit09, Espinosa10, Graversen01}. Several of these papers have considered the problem of predicting the ultimate supremum of a Brownian motion (with drift), also with a finite time horizon, and of the special case of a spectrally positive stable L\'evy process.    

Before discussing the contents of this paper in more detail let us note some facts about L\'evy processes (all of which can be found in \cite{Kyprianou14}, see also \cite{Bertoin96} and \cite{Sato99} e.g.). Recall that a L\'evy process has stationary independent increments and that its law is characterised by a triplet $(\gamma,\sigma,\Pi)$, where $\gamma \in \R$, $\sigma \geq 0$ and $\Pi$ is a measure on $\R \setminus \{0\}$ satisfying the integrability condition
\[ \int_\R (1 \wedge x^2) \Pi(\mathrm{d}x)<\infty. \]
It is spectrally negative when it has no positive jumps i.e. $\Pi(\R_{>0})=0$ and when $-X$ is not a subordinator.

For the spectrally negative L\'evy process $X$, its Laplace exponent
\be{eq_psi} 
\psi(z) = \log \left( \E[e^{zX_1}] \right) = \frac{\sigma^2}{2}z^2 -\gamma z + \int_{(-\infty,0)} \left( e^{zx}-1-zx\mathbf{1}_{\{ x>-1\}} \right) \Pi(\mathrm{d}x) 
\ee
is finite at least for all $z \geq 0$. Further $\psi$ is infinitely often differentiable and strictly convex on the interior of its domain (which is necessarily an interval), with $\psi(0)=0$ and $\psi(\infty)=\infty$. Hence we can define its right inverse as
\[ \Phi(q) = \sup \{ z \geq 0 \, | \, \psi(z)=q \} \quad \text{for all $q \geq 0$.} \]
Further it is well known that 
\be{eq_ult_sup}
\overline{X}_\infty<\infty \text{ a.s.} \iff \lim_{t \to \infty} X_t=-\infty \text{ a.s.} \iff \Phi := \Phi(0)>0. 
\ee
If the conditions in \eqref{eq_ult_sup} hold, which we will assume throughout, we have that $\overline{X}_\infty \sim \text{Exp}(\Phi)$ and $\psi'(0+)=\E[X_1] \in [-\infty,0)$. The condition $\E[\varphi(\overline{X}_\infty)]<\infty$ hence boils down to
\be{eq_cond_varphi}
\int_0^\infty \varphi(x) e^{-\Phi x} \d x<\infty.
\ee

For any $y \geq 0$, we define the process $Y^y=(Y^y_t)_{t \geq 0}$ as the process $X$ reflected in its running supremum, started from level $y$ i.e.
\be{eq_def_Y} 
Y^y_t = (y \vee \overline{X}_t) - X_t \quad \text{for all $t \geq 0$.} 
\ee
It is well known that $Y^y$ is a strong Markov process. Note that under the standing assumption \eqref{eq_ult_sup} we have that $Y^y_t \to \infty$ a.s. as $t \to \infty$. An interpretation of $y>0$ in \eqref{eq_def_Y} is that the process $X$ was started at some time prior to $t=0$ and that at $t=0$ we observe the reflected distance to be $y=\overline{X}_0-X_0$.

The rest of this paper is organised as follows. In Section \ref{sec_main} we state and discuss our main results. We first use standard arguments to express the optimal prediction problem \eqref{eq_main} as an equivalent optimal stopping problem driven by the above process $Y^y$. After showing that the problem is trivial when $\varphi$ is non-decreasing we turn to the case of the quadratic penalty function \eqref{def_pen}, and fully characterise and discuss the solution for that case in Theorem \ref{thm_main}. Section \ref{sec_proofs} contains a proof of Theorem \ref{thm_main} (which is a bit of work).  Finally, in Section \ref{sec_exam} we look at some specific examples of spectrally negative L\'evy processes and further illustrate the results from Theorem \ref{thm_main}.

\section{Main results and discussion}\label{sec_main}

Recall that throughout we assume that $X$ satisfies the conditions in \eqref{eq_ult_sup} and that \eqref{eq_cond_varphi} holds. In the first lemma we show, using standard arguments, that the optimal prediction problem \eqref{eq_main} can be expressed as an optimal stopping problem driven by the process $Y^0$.

\begin{lemma}\label{lem_rewr} Let the non-negative function $H$ be given by
\be{def_H} 
H(y) = \varphi(y) \left( 1-e^{-\Phi y} \right) + \int_y^\infty \varphi(z)\Phi e^{-\Phi z} \d z \quad \text{for all $y \geq 0$} 
\ee
and further define the function $V$ as
\be{def_V} 
V(y) = \inf_\tau \E[H(Y^y_\tau)] \quad \text{for all $y \geq 0$} 
\ee
where the infimum is taken over all stopping times with respect to the naturally enlarged filtration generated by $Y^y$.

Then $V(0)$ equals \eqref{eq_main}.
\end{lemma}

\begin{proof} Since $\varphi$ is non-negative, we can assume in the below w.l.o.g. that $\varphi$ is bounded (the result can be extended using monotone convergence otherwise). We closely follow the proof of Lemma 1 in \cite{DuToit08}, first to establish for any $t \geq 0$ that
\[ \E \left[ \left. \varphi \left( \overline{X}_\infty-X_t \right) \, \right| \, \mathcal{F}_t \right] = \E \left[ \varphi \left( y+ (\overline{X}_\infty-y)^+ \right) \right] \Big|_{y=\overline{X}_t-X_t}. \]
Using that $\overline{X}_\infty \sim \text{Exp}(\Phi)$ (cf. \eqref{eq_ult_sup}), a straightforward computation shows that the above right hand side equals $H(\overline{X}_t-X_t)$. An application of Hunt's Lemma (cf. e.g. E14.1 in \cite{Williams91}) now yields the result.
\end{proof}

Note that for $y>0$, we can understand $V(y)$ to be equivalent to \eqref{eq_main} in the situation that $X$ was started at some time prior to $t=0$ and at $t=0$ we observe the reflected distance to be $y=\overline{X}_0-X_0$.

\begin{remark}\label{rem_stopp_inf} We only consider penalty functions $\varphi$ for which $\varphi(\infty)$ exists (possibly infinite). Then $H(\infty)$ exists as well, and since both $X$ and $Y^y$ have well defined limits as $t \to \infty$ we can allow for $[0,\infty]$ valued stopping times both in \eqref{eq_main} and \eqref{def_V}.
\end{remark}

It is not hard to show that if the penalty function $\varphi$ is non-decreasing then it is for $y=0$ optimal to stop immediately. Apparently the homogeneity of $X$, also keeping in mind that $X$ drifts to $-\infty$, guarantees that waiting is suboptimal and that \eqref{eq_main} is equal to $\E[\varphi (\overline{X}_\infty)]$.

\begin{lemma}\label{lem_varphi_nondecr} If $\varphi$ is non-decreasing then $\tau=0$ is optimal in \eqref{eq_main}.
\end{lemma}

\begin{proof} By virtue of Lemma \ref{lem_rewr} above, since $Y^0_0=0$ it suffices to show that $H$ is non-decreasing. For this, take any $0 \leq y_1 \leq y_2$ and note that integration by parts yields
\[ \int_{y_1}^{y_2} \varphi(z) \Phi e^{-\Phi z} \d z = -\varphi(y_2) e^{-\Phi y_2}+\varphi(y_1) e^{-\Phi y_1} + \int_{y_1}^{y_2} e^{-\Phi z} \mathrm{d}\varphi(z) \]
so that 
\[ H(y_2)-H(y_1) = \varphi(y_2) - \varphi(y_1) - \int_{y_1}^{y_2} e^{-\Phi z} \mathrm{d}\varphi(z). \]
Since $\varphi$ is non-decreasing we have that
\[ \int_{y_1}^{y_2} e^{-\Phi z} \mathrm{d}\varphi(z) \leq \int_{y_1}^{y_2} \mathrm{d}\varphi(z) = \varphi(y_2) - \varphi(y_1) \]
and hence indeed $H(y_2)-H(y_1) \geq 0$.
\end{proof}

\begin{remark} Of course, the above Lemmas \ref{lem_rewr} \& \ref{lem_varphi_nondecr} do not rely on $X$ being spectrally negative. In particular, $\tau=0$ is optimal in \eqref{eq_main} for any L\'evy process whose ultimate supremum is finite a.s.
\end{remark}

Now we turn to the more interesting case of a quadratic penalty function of the form 
\be{def_quad_pen} 
\varphi(x) =(x-b)^2 \quad \text{for some $b > 0$,}
\ee
in which case $H$ as defined in \eqref{def_H} boils down to
\be{def_quad_H}
H(y) = (y-b)^2+\frac{2}{\Phi} e^{-\Phi y} \left( y-b + \frac{1}{\Phi} \right).
\ee 
Note that the shape of $\varphi$ is preserved in $H$ in the sense that $H'$ is $<0$, $=0$, $>0$ for $y \in (0,b)$, $y=b$, $y>b$ resp. with $H(0)<\infty$ and $H(\infty)=\infty$. Hence it makes sense to expect that in the optimal stopping problem \eqref{def_V} it is optimal to stop when $Y^y$ is close to $b$ i.e. when the distance between the running supremum and the position of $X$ is close to $b$.

However, keeping in mind that $Y^y$ drifts to $\infty$, for any $y \in (b,\infty)$ it is not obvious whether it is better to stop and accept the payoff $H(y)>H(b)$ or to wait until $Y^y$ moves closer to $b$ while risking that it drifts only further away from $b$. Further, for any $y \in [0,b)$ there is a similar dilemma, where waiting comes for $Y^y$ to move closer to $b$ comes with the risk of experiencing a positive jump taking $Y^y$ (far) over the level $b$ to more unfavourable payoffs. So the structure of the solution is not very easy to guess.

Before presenting our main result we make two assumptions in addition to the conditions in \eqref{eq_ult_sup}, namely:
\begin{enumerate}
\item[(A1)] there exists $z_0<0$ such that $\psi(z_0)<\infty$,
\item[(A2)] if $X$ has bounded variation then $\Pi$ has no atoms.
\end{enumerate}
Recall that (A1) is a restriction on the `large jumps' in the sense that it is equivalent to 
\[ \int_{(-\infty,-1)} e^{z_0 x} \Pi(\mathrm{d}x)<\infty \]
(cf. Theorem 3.8 in \cite{Kyprianou14}). It guarantees that the domain of $\psi$ (on the interior of which it is infinitely often differentiable and strictly convex) contains $[z_0,\infty)$. In particular it implies that all the integer moments of $X_1$ are finite, and $\psi^{(k)}(0)$ equals the $k$-th cumulant of $X_1$. Assumption (A2) is needed to exclude some pathological cases in the proofs.

Further recall that there exist families of scale functions denoted $W^{(q)}: \R \to [0,\infty)$ and $Z^{(q)}: \R \to [1,\infty)$ for $q \geq 0$ associated with $X$, where $W^{(q)}(x)=0$ for $x<0$ while on $[0,\infty)$ it is continuous, strictly increasing and uniquely characterised by
\[ \int_0^\infty e^{-\beta x} W^{(q)}(x) \d x = \frac{1}{\psi(\beta)-q} \quad \text{for $\beta>\Phi(q)$,} \]
and
\[ Z^{(q)}(x) = 1+q \int_0^x W^{(q)}(y) \d y \quad \text{for $x \in \R$.} \]
For simplicity we denote $W:=W^{(0)}$ and $Z:=Z^{(0)}$. These scale functions are commonly used to express quantities involving one- and two-sided exit problems for $X$, see e.g. Chapter 8 in \cite{Kyprianou14}. Explicit expressions (or numerical algorithms) exist for many cases, cf. e.g. \cite{Hubalek10} (see also Section \ref{sec_exam} below).

Here is our main result, which fully characterises the solution to the optimal stopping problem \eqref{def_V}.  

\begin{theorem}\label{thm_main} Consider the optimal stopping problem \eqref{def_V} for $H$ given by \eqref{def_quad_H}. Suppose that assumptions (A1) and (A2) above hold in addition to the conditions in \eqref{eq_ult_sup}. Further we denote for $y \geq 0$ and $a \geq 0$ the stopping time
\[ \tau^y_a = \inf \{ t \geq 0 \, | \, Y^y_t \geq a \} \]
and we set
\[ I_1(a)=\int_0^a W(x) \d x, \quad I_2(a)=\int_0^a xW(x) \d x \quad \text{and} \quad
I_3(a)=\int_0^a e^{-\Phi x} W(x) \d x. \]
We have the following two cases.
\begin{enumerate}[label=(\roman*)]
\item If
\[ b \leq \frac{\psi'(\Phi)/\Phi-\psi''(0)/2}{\psi'(0)} \]
then for any $y \geq 0$ it holds that $\tau=0$ is optimal in \eqref{def_V} and hence $V=H$.
\item If 
\[ b > \frac{\psi'(\Phi)/\Phi-\psi''(0)/2}{\psi'(0)} \]
then there exists $a^* \in (0,b]$ so that for any $y \geq 0$ the stopping time $\tau^y_{a^*}$ is optimal in \eqref{def_V}. A characterisation of $a^*$ is given in Lemma \ref{lem_g}. Further $V$ is continuous and can be expressed as
\[ V(y) = \begin{cases}
H(y) - I_1(a^*-y) \big( 2\psi'(0)(b-y)+\psi''(0) \big) & \\
\qquad +2\psi'(0)I_2(a^*-y)+2\psi'(\Phi)e^{-\Phi y} I_3(a^*-y)/\Phi & \text{if $y \in [0,a^*)$} \\
H(y) & \text{if $y \in [a^*,\infty)$.}
\end{cases} \]
\end{enumerate}
\end{theorem}

A proof is provided in Section \ref{sec_proofs}. Case (i) is covered by Lemmas \ref{lem_CS} \& \ref{lem_opt2}, while for case (ii) Lemmas \ref{lem_CS} \& \ref{lem_opt1} give the optimal stopping time and the expression for $V$ follows from Lemma \ref{lem_fa}.

This result shows that in both cases (i) and (ii), if the process $Y^y$ starts in $(b,\infty)$ (i.e. $y \in (b,\infty)$) then it is best to stop immediately. That is to say, it is not worth waiting for $Y^y$ to decrease closer to $b$ due to the risk of it moving upwards rather. Further in case (ii), if the process $Y^y$ starts in $(0,a^*)$ then it is optimal to wait until $Y^y$ gets closer to $b$ and to stop only when it first enters $[a^*,\infty)$. Note that it is indeed very well possible that $a^*<b$ (see also Section \ref{sec_exam}), in this situation it is optimal to stop when $Y^y$ is close enough to $b$ rather than waiting until $Y^y$ actually hits the level $b$ and risking that a positive jump takes $Y^y$ (deep) into $(b,\infty)$.

The distinction between the two cases in the result is also interesting. It can be verified that 
\be{eq_9Apr1} 
\frac{\psi'(\Phi)/\Phi-\psi''(0)/2}{\psi'(0)} 
\ee
is non-negative, and that it vanishes if and only if $X$ has no jumps i.e. when $X$ is a Brownian motion with negative drift (cf. Lemma \ref{lem_bound}). Indeed, if $X$ has no jumps then neither has $Y^y$ and therefore it is obvious that while $Y^y$ lives in $[0,b)$ it is best to wait for it to hit $b$ (which it will a.s.). The above result confirms this since case (ii) always applies in this situation (with $a^*=b$). On the other hand, if $X$ does have jumps then $b$ has to be large enough (or: the difference between $H(y)$ for $y$ close to $0$ and $H(b)$ has to be large enough) for the benefit of waiting for $Y^y$ to get closer to $b$ to outweigh the risk of a jump taking $Y^y$ (far) beyond $b$.

Finally it is worth noting that we would expect that the above result remains valid under the weaker condition that only the first moment of $X_1$ is guarateed to be finite. It would then seem that the same two cases remain present as long as the second moment is also finite, but that when the second moment is infinite then \eqref{eq_9Apr1} is also infinite due to $\psi'(0)<0$ and $\psi''(0)=\var(X_1)=\infty$ and hence only the first case is present. This is of course due to the quadratic form of the penalty function.

\section{Remaining proofs}\label{sec_proofs}

This section contains a proof of Theorem \ref{thm_main}, broken down into a number of lemmas. We use the notation introduced above and in the statement of Theorem \ref{thm_main} throughout.

The proof is a bit hairy since the optimal stopping problem \eqref{def_V} involves a payoff function $H$ that is not monotone or convex/concave (often for payoff functions with such properties more straightforward arguments are possible) and further since $H(Y^y_t) \to \infty$ a.s. as $t \to \infty$ the integrability conditions that are typically assumed to formulate general results in optimal stopping theory do not hold in this case.

Recall that we are working under the assumption that the conditions in \eqref{eq_ult_sup} hold, and that (A1) and (A2) from Section \ref{sec_main} hold. Using (A1), in the interior of the domain of $\psi$, which contains the point $0$, we get from \eqref{eq_psi} that
\begin{multline}\label{eq_psi_der}
\psi'(z)=\sigma^2 z-\gamma+\int_{(-\infty,0)} \left( xe^{zx}-x\mathbf{1}_{\{ x>-1 \}} \right) \Pi(\mathrm{d}x), \\ 
\psi''(z)=\sigma^2+\int_{(-\infty,0)} x^2 e^{zx} \Pi(\mathrm{d}x) \quad \text{and} \quad \psi'''(z)=\int_{(-\infty,0)} x^3 e^{zx} \Pi(\mathrm{d}x).
\end{multline}

Briefly returning to the scale functions introduced in Section \ref{sec_main}, from Lemma 8.6 in \cite{Kyprianou14} we know that $W(0)=0$ if $X$ has unbounded variation and $W(0)>0$ otherwise. Further, using (A2) we know that $W \in C^1(0,\infty)$ and $W'>0$ on $(0,\infty)$, and the right derivative in $0$ exists which with a slight abuse of notation we denote by $W'(0) \in (0,\infty]$ (cf. Lemma 8.2 and the discussion following it, and Exercise 8.5, both in \cite{Kyprianou14}). 

Further, for twice continuously differentiable functions $f:\R\to\R$ we define 
\be{Afx} 
\mathcal{A}f(y) = \frac{\sigma^2}{2} f''(y)+\gamma f'(y) + \int_{\R} \big( f(y+x)-f(y)+x f'(y) \mathbf{1}_{\{ |x|<1 \}} \big) \Pi(\mathrm{d}x). 
\ee
If $f$ and $\Pi$ are such that the integral in \eqref{Afx} is finite, It\^o's formula (cf. e.g. Theorem 4.3 in \cite{Kyprianou14}) yields for all $t \geq 0$
\begin{align}
f(X_t) &= f(X_0)+ \int_0^t f'(X_{s^-})\d X_s+\frac{\sigma^2}{2}\int_0^tf''(X_s)\d s \notag\\
& \qquad +\int_{[0,t]}\int_{\R}\left(f(X_{s^-}+x)-f(X_{s^-})-xf'(X_{s^-})\right)N(\mathrm{d} s\times \mathrm{d} x) \notag\\
&=f(X_0) + M_t + \int_0^t \mathcal{A}f(X_{s^-}) \d s, \label{K_decompo}
\end{align}
where $N$ is the Poisson random measure associated with the jumps of $X$ and
\begin{align*}
M_t=&\int_0^tf'(X_{s^-})\d \left(X_s+\gamma s-\int_{[0,s]}\int_{|x|\geq 1}xN(\mathrm{d}u\times \mathrm{d}x)\right)\\
& \qquad +\int_{[0,t]}\int_{\R}\left(f(X_{s^-}+x)-f(X_{s^-})-x\mathbf{1}_{\{|x|<1\}}f'(X_{s^-})\right)N(\mathrm{d} s\times \mathrm{d} x)\\
& \qquad -\int_{[0,t]}\int_{\R}\left(f(X_{s^-}+x)-f(X_{s^-})-x\mathbf{1}_{\{|x|<1\}}f'(X_{s^-})\right)\Pi(\mathrm{d}x) \d s \quad \text{for all $t \geq 0$.}
\end{align*}
By the L\'evy-It\^o decomposition the first integral in the expression for $M$ is a local martingale as its integrator is a martingale. Under the condition 
\[\E\left[\int_{[0,t]}\int_{\R}f(X_{s^-}+x)-f(X_{s^-})-x\mathbf{1}_{\{|x|<1\}}f'(X_{s^-})\Pi(\mathrm{d}x) \d s\right]< \infty\]
the Compensation Formula (see e.g. Theorems 4.3 \& 4.4 and Corollary 4.6 in \cite{Kyprianou14}) can be invoked to conclude that $M$ is a local martingale.

Now, the first three lemmas below make up some preparation.

\begin{lemma}\label{lem_bound} The quantity
\be{9Apr2} 
\frac{\psi'(\Phi)/\Phi-\psi''(0)/2}{\psi'(0)} 
\ee
vanishes when $\Pi(\R_{<0})=0$ and is strictly positive otherwise.
\end{lemma}

\begin{proof} If $\Pi(\R_{<0})=0$ then after some straightforward algebra we find that $\Phi=2\gamma/\sigma^2$ and, using \eqref{eq_psi_der}, indeed that \eqref{9Apr2} vanishes. Now suppose that $\Pi(\R_{<0})>0$. From \eqref{eq_psi_der} it follows that $\psi'''<0$ i.e. $\psi'$ is strictly concave, hence
\be{9Apr3} 
\psi''(0)> \frac{\psi'(\Phi)-\psi'(0)}{\Phi}. 
\ee
Further, using \eqref{eq_psi_der} to evaluate $\psi'(0)+\psi'(\Phi)$ and simplifying somewhat using that $\psi(\Phi)=0$ yields
\[ \psi'(0)+\psi'(\Phi)=\int_{(-\infty,0)} \left( x+xe^{\Phi x}-\frac{2}{\Phi} \left( e^{\Phi x}-1 \right) \right) \Pi(\mathrm{d}x). \]
Standard analysis shows that the integrand is strictly negative on $(-\infty,0)$ and hence $\psi'(0)+\psi'(\Phi)<0$. Using this together with \eqref{9Apr3} yields
\[ \frac{\psi'(\Phi)}{\Phi}-\frac{\psi''(0)}{2}<\frac{\psi'(\Phi)}{\Phi}-\frac{\psi'(\Phi)-\psi'(0)}{2\Phi}=\frac{\psi'(\Phi)+\psi'(0)}{2\Phi}<0, \]
and since $\psi'(0)<0$ the result indeed follows.
\end{proof}

\begin{lemma}\label{lem_fa} For any $a \geq 0$, define the function $f_a$ as 
\be{def_f_a} 
f_a(y)=\E \left[ H\left( Y^y_{\tau_a^y} \right) \right] \quad \text{for $y \geq 0$}. 
\ee
Then we have that
\[ f_a(y)=\begin{cases}
H(y)-I_1(a-y) \big( 2\psi'(0)(b-y)+\psi''(0) \big) +2\psi'(0)I_2(a-y) & \\
\quad +2\psi'(\Phi)e^{-\Phi y}I_3(a-y)/\Phi +g(a)W(a-y)/W'(a) & \text{if $y \in [0,a)$}, \\
H(y) & \text{if $y \in [a,\infty)$}
\end{cases} \]
where the function $g:[0,\infty) \to \R$ is given by
\[ g(a)= 2\psi'(0)I_1(a)-2\psi'(\Phi)I_3(a)\\
    +W(a)\left(\psi''(0)-2a\psi'(0)+2b\psi'(0)-\frac{2\psi'(\Phi)}{\Phi}e^{-\Phi a}\right) \]
for all $a \geq 0$.
\end{lemma}

\begin{proof} For any $y \geq a$ we have $\tau_a^y=0$ a.s. and hence the result is obvious. Next fix $y \in [0,a)$. Set
\[ K(t)=\E\left[e^{-t Y^y_{\tau_a^y}}\right]. \]
From Theorem 8.10 in \cite{Kyprianou14} we have that
\be{10Apr1} 
K(t)=e^{-ty} \left( Z_t^{(-\psi(t))}(a-y)-W_t^{(-\psi(t))}(a-y)\left(\frac{-\psi(t)W_t^{(-\psi(t))}(a) +tZ_t^{(-\psi(t))}(a)}{{W_t^{(-\psi(t)) \prime}}(a)+tW_t^{(-\psi(t))}(a)} \right) \right) 
\ee
for all $t$ so that $\psi(t)<\infty$, which by (A1) holds on some open interval containing $\R_{\geq 0}$. Here $W_t$ and $Z_t$ are the scale functions associated with $X$ after a change of measure, however using Lemma 8.4 in \cite{Kyprianou14} we can simply write
\begin{multline*} 
W_t^{(-\psi(t))}(x)=e^{-tx}W^{(-\psi(t)+\psi(t))}(x)=e^{-tx}W(x) \quad \text{and} \\ 
Z_t^{(-\psi(t))}(x) = 1-\psi(t) \int_0^x e^{-tz}W(z) \d z. 
\end{multline*}

Plugging the expression for $H$ from \eqref{def_quad_H} into \eqref{def_f_a} we find that we can write
\be{10Apr2} 
f_a(y)=K''(0)+2bK'(0)-\frac{2}{\Phi}K'(\Phi)+\frac{2}{\Phi}\left(\frac{1}{\Phi}-b\right)K(\Phi)+b^2. 
\ee
Now it is a matter of some tedious algebra to work this out using \eqref{10Apr1}. Defining for notational convenience
\[ Z(t,x)=Z_t^{(-\psi(t))}(x), \quad Z_1(t,x)=\frac{\partial Z(t,x)}{\partial t} \quad \text{and} \quad Z_2(t,x)=\frac{\partial^2 Z(t,x)}{\partial t^2} \]
we can compute (also using that $\psi(\Phi)=0$)
\[ K(\Phi)=e^{-\Phi y}-\Phi\frac{W(a-y)}{W'(a)} \]
and
\[ K'(\Phi)=-ye^{-\Phi y}+e^{-\Phi y}Z_1(\Phi,a-y)-\frac{W(a-y)}{W'(a)}\left(-\psi'(\Phi)e^{-\Phi a}W(a)+1+\Phi Z_1(\Phi,a)\right). \]
Further, also using that $\psi(0)=0$,
\[ K'(0)=-y+Z_1(0,a-y)-\frac{W(a-y)}{W'(a)}\left( -\psi'(0)W(a)+1 \right) \]
and
\begin{multline*} 
K''(0)=y^2-2yZ_1(0,a-y)+Z_2(0,a-y)\\
-\frac{W(a-y)}{W'(a)} \left( -\psi''(0)W(a)+2a\psi'(0)W(a)+2Z_1(0,a) \right). 
\end{multline*}
Plugging all these elements back into \eqref{10Apr2} and simplifying a bit yields the result.
\end{proof}

\begin{remark} Theorem 8.10 in \cite{Kyprianou14} uses the stopping time $\sigma_a^y := \inf \{ t >0 \, | \, Y^y_t > a\}$ rather than $\tau_a^y$ to formulate \eqref{10Apr1}. However for $y<a$ these stopping times are a.s. equal. Indeed, the event $\{ \sigma_a^y>\tau_a^y \}$ consists of paths of $Y^y$ that first enter $[a,\infty)$ by hitting $a$ and then take a strictly positive amount of time after that to enter $(a,\infty)$. The former behaviour requires that $X$ has unbounded variation (see Exercise 7.6 in \cite{Kyprianou14}) while the latter requires that $X$ has bounded variation (see Theorem 6.5 in \cite{Kyprianou14}).
\end{remark}

\begin{lemma}\label{lem_g} Consider again the function $g$ defined in Lemma \ref{lem_fa}. 
\begin{enumerate}[label=(\roman*)]
\item If
\[ b \leq \frac{\psi'(\Phi)/\Phi-\psi''(0)/2}{\psi'(0)} \]
then $g>0$ on $(0,\infty)$.
\item If 
\[ b > \frac{\psi'(\Phi)/\Phi-\psi''(0)/2}{\psi'(0)} \]
then $g$ has a unique root in $(0,\infty)$ denoted $a^*$ and we have that $g<0$ resp. $g>0$ on $(0,a^*)$ resp. $(a^*,\infty)$.
\end{enumerate}
\end{lemma}

\begin{proof} Note that we may write $g(a)=g_1(a)+W(a)g_2(a)$ where
\[ g_1(a) := 2\psi'(0)I_1(a)-2\psi'(\Phi)I_3(a) \quad \text{and} \quad
    g_2(a) :=\psi''(0)+2\psi'(0)(b-a)-\frac{2e^{-\Phi a}}{\Phi}\psi'(\Phi). \]
Since $\Phi>0$, $\psi'(0)<0$ and $\psi'(\Phi)>0$ we see that $g_1(a)<0$ for $a>0$ and $g_1(0)=0$. Moreover, $g_2'(a)=-2\psi'(0)+2\psi'(\Phi)e^{-\Phi a}>0$ i.e. $g_2$ is strictly increasing.

If $a^* \in (0,\infty)$ exists so that 
\[ g(a^*)=g_1(a^*)+W(a^*)g_2(a^*)=0 \]
then, using that $g'(a)=W'(a)g_2(a)$ and $g_1<0$
\be{12Apr1} 
g'(a^*)=W'(a^*)g_2(a^*)=W'(a^*) \cdot -\frac{g_1(a^*)}{W(a^*)}>0. 
\ee
It follows that $g$ has at most one root on $(0,\infty)$.

For existence of a root we first show that 
\be{10Apr5}
\lim_{a \to \infty} g(a)=\infty.
\ee
For this, note that
\begin{align*}
    \lim_{a\rightarrow\infty}\frac{g(a)}{W(a)}&=\lim_{a\rightarrow\infty} \frac{g_1(a)}{W(a)}+g_2(a) \\
    &=\lim_{a\rightarrow\infty} 2\psi'(0)\frac{\int_0^aW(x) \d x}{W(a)}-2\psi'(\Phi)\frac{\int_0^ae^{-\Phi x}W(x) \d x}{W(a)}+g_2(a) \\
    &\geq \lim_{a\rightarrow\infty} 2(\psi'(0)-\psi'(\Phi))\int_0^a\frac{W(x)}{W(a)} \d x+g_2(a) \\
        &= \lim_{a\rightarrow\infty} 2(\psi'(0)-\psi'(\Phi))\int_0^a\frac{W(a-x)}{W(a)} \d x+g_2(a) \\
        &= \infty,
\end{align*}
the final equality since $g_2(a) \to \infty$ as $a \to \infty$ (recall that $\psi'(0)<0$) while from Theorems 8.1 \& 3.12 in \cite{Kyprianou14} we can deduce that
\[ \frac{W(a-x)}{W(a)} \leq \P \left( \overline{X}_\infty>x \right) =e^{-\Phi x}. \]
Since $W(a)>0$ for $a>0$ and $W$ is strictly increasing, \eqref{10Apr5} indeed follows.

It remains to look at $g(a)$ for $a$ close to $0$. First consider the case that $X$ has unbounded variation. Then $g(0)=0$ since $W(0)=0$ and $g_1(0)=0$. Further, since $g'(0+)=W'(0)g_2(0)$ and $W'(0)>0$ it follows that $g$ has a unique root on $(0,\infty)$ if and only if $g_2(0)<0$ i.e. if and only if
\be{10Apr7}
b > \frac{\psi'(\Phi)/\Phi-\psi''(0)/2}{\psi'(0)}
\ee
(indeed, if $g_2(0)=0$ then $g_2>0$ on $(0,\infty)$ since $g_2$ is strictly increasing). If $X$ has bounded variation then $g(0)=W(0)g_2(0)$ where $W(0)>0$ and hence we can again conclude that $g$ has a unique root on $(0,\infty)$ if and only if $g_2(0)<0$ i.e. if and only if \eqref{10Apr7} holds.
\end{proof}

Now we are ready to start working towards the proof of Theorem \ref{thm_main}. Note that it is obvious that $V$ as defined in \eqref{def_V} is bounded below by $H(b)$ (since $H$ is) and that $V \leq H$ (by plugging $\tau=0$ into the right hand side of \eqref{def_V}).

\begin{lemma}\label{lem_CS}
Define the following regions in the state space $\R_{\geq 0}$ of $Y^y$:
\[ \mathcal{C} = \{ y\in \R_{\geq 0} \, | \, V(y)<H(y) \} \quad \text{and} \quad \mathcal{S} = \{ y\in \R_{\geq 0} \, | \, V(y)=H(y) \}=\mathcal{C}^c. \]
Then for any $y\geq0$ the stopping time
\[ \tau^*:=\inf\{t\geq0 \, | \,  Y_t^y\in \mathcal{S} \} \]
is optimal in \eqref{def_V}, i.e. $\mathcal{C}$ and $\mathcal{S}$ are the continuation and stopping region respectively. Further $V$ is continuous.
\end{lemma}

\begin{proof} For continuity of $V$, fix some $y_0 \geq 0$. Let $\{y_k\}_{k \geq 1}$ be a sequence so that $y_k \downarrow y_0$ as $k \to \infty$. Fix any $\eps>0$. By definition of the infimum in \eqref{def_V} there exists a stopping time $\tau_\eps$ so that 
\be{3jan1} 
V(y_0) \geq \E[H(Y^{y_0}_{\tau_\eps})] - \eps. 
\ee
Note that since $V(y_0) \leq H(y_0)<\infty$, this implies that 
\[ \E[H(Y^{y_0}_{\tau_\eps})]<\infty, \]
and since $H(\infty)=\infty$ and $Y^{y_0}_\infty=\infty$ it also follows that $\tau_\eps<\infty$ a.s. Further, it is clear from the expression for $H$ that $C>0$ exists so that $H(y) \geq Cy$ for all $y \geq 0$ and hence it also follows that
\be{3jan2} 
\E[Y^{y_0}_{\tau_\eps}] \leq \frac{1}{C} \E[H(Y^{y_0}_{\tau_\eps})] <\infty. 
\ee
Now, for any $k \geq 1$ we trivially have that
\be{3jan3} 
V(y_k) \leq \E[H(Y^{y_k}_{\tau_\eps})]. 
\ee
Noting that $|H'|$ can be bounded above by $y \mapsto Ay+B$ for some $A,B>0$ and that $|Y^{y_k}_{\tau_\eps}-Y^{y_0}_{\tau_\eps}| \leq y_k-y_0$, an application of Taylor's Theorem shows that
\begin{align*} 
\E \left[ \big| H(Y^{y_k}_{\tau_\eps})-H(Y^{y_0}_{\tau_\eps}) \big| \right] &\leq (A \E[Y^{y_k}_{\tau_\eps}]+B) (y_k-y_0) \\
&\leq  \big( A (\E[Y^{y_0}_{\tau_\eps}]+(y_k-y_0))+B \big) (y_k-y_0). 
\end{align*}
Letting $k \to \infty$, since $y_k \downarrow y_0$ and recalling \eqref{3jan2} we see that the ultimate rhs above vanishes and hence we find that
\[ \E[H(Y^{y_k}_{\tau_\eps})] \longrightarrow \E[H(Y^{y_0}_{\tau_\eps})]. \]
We can now conclude that
\[ \limsup_{k \to \infty} V(y_k) \stackrel{\eqref{3jan3}}{\leq} \limsup_{k \to \infty} \E[H(Y^{y_k}_{\tau_\eps})] = \E[H(Y^{y_0}_{\tau_\eps})] \stackrel{\eqref{3jan1}}{\leq} V(y_0)+\eps. \]
Since $\eps>0$ was arbitrary, it in fact follows that
\[ \limsup_{k \to \infty} V(y_k) \leq V(y_0). \]

Finally, in a similar way it can be shown that 
\[ \liminf_{k \to \infty} V(y_k) \geq V(y_0) \]
and hence right continuity of $V$ in $y_0$ follows. Left continuity (provided that $y_0>0$) can be shown using analogue arguments.

It follows from Theorem 6 in \cite{Shiryaev08} (recall that $H$ is non-negative) that 
\[ \inf\{t\geq0 \, | \, V(Y_t^y)=H(Y_t^y)\} \] 
is an optimal stopping time for \eqref{def_V}. 
\end{proof}

\begin{remark}\label{rem_XeqY} Note that for any $y \geq 0$, since $X$ is spectrally negative its running supremum $\overline{X}$ is a continuous process. In particular, on the time interval $[0,\kappa]$ where
\[ \kappa = \inf \{ t \geq 0 \, | \, X_t \geq y \} = \inf \{ t \geq 0 \, | \, \overline{X}_t = y \} = \inf \{ t \geq 0 \, | \, Y^y_t =0 \} \]
it holds that $Y^y_t=y-X_t$ and the filtrations generated by $Y^y$ and $X$ coincide.
\end{remark}

\begin{lemma}\label{lem_Z^y} Fix any $y \geq 0$. Let $Z^y=(Z^y_t)_{t \geq 0}$ be the process given by $Z^y_t = H(y-X_t)$ for all $t \geq 0$. Further define the stopping times
\[ \kappa_y^{\widehat{y}} = \inf \{ t \geq 0 \, | \, X_t \geq y-\widehat{y} \} \quad \text{for $\widehat{y} \geq 0$}. \] 
We have the following.
\begin{enumerate}
\item[(i)] Suppose that 
\[ b \leq \frac{\psi'(\Phi)/\Phi-\psi''(0)/2}{\psi'(0)}. \]
Then the stopped process $\left(Z^y_{t \wedge \kappa_y^0}\right)_{t \geq 0}$ is a local submartingale.
\item[(ii)] Next suppose that 
\[ b > \frac{\psi'(\Phi)/\Phi-\psi''(0)/2}{\psi'(0)}. \]
Then there exists $y^* \in (0,a^* \wedge b)$ (independent of $y$) such that the stopped process $\left(Z^y_{t \wedge \kappa_y^{ y^*}}\right)_{t \geq 0}$ is a local submartingale.
\end{enumerate}
(Recall Lemma \ref{lem_bound} and that $a^*$ was defined in Lemma \ref{lem_g}).
\end{lemma}

\begin{proof} Note that $H$ can be extended in a smooth way beyond $0$ since $H''(0+)=-2b\Phi$. Let for notational convenience the function $f^y$ be given by $f^y(x)=H(y-x)$. Applying the operator defined in \eqref{Afx} we have
\begin{multline}\label{Kevin30} 
\mathcal{A}f^y(l) = \frac{\sigma^2}{2} H''(y-l)+\gamma H'(y-l) \\
+ \int_{(-\infty,0)} \big( H(y-l-x)-H(y-l)+x H'(y-l) \mathbf{1}_{\{ x>-1 \}} \big) \Pi(\mathrm{d}x)
\end{multline} for $l\leq y$.
Using Taylor's Theorem and the fact that $H''$ is bounded it is easy to see that for some $C_1>0$
\[ \int_{(-1,0)} \big| H(y-l-x)-H(y-l)+x H'(y-l) \big| \Pi(\mathrm{d}x) \leq C_1 \int_{(-1,0)} x^2 \Pi(\mathrm{d}x). \]
Again Taylor's Theorem together with $|H'(x)| \leq C_2 x$ for some $C_2>0$ gives that
\begin{multline*} 
\int_{(-\infty,-1]} \big| H(y-l-x)-H(y-l) \big| \Pi(\mathrm{d}x) \\
\leq C_2 \int_{(-\infty,-1]} x^2 \Pi(\mathrm{d}x) +C_2 |y-l| \int_{(-\infty,-1]} |x| \Pi(\mathrm{d}x),
\end{multline*}
where both integrals in the right hand side are finite on account of assumption (A1). Using the above two inequalities together with Fubini's Theorem it follows that 
\begin{align*}\E\left[\int_{[0,t]}\int_{(-\infty,0)}\left(H(y-X_{s^-}-x)-H(y-X_{s^-})+x\mathbf{1}_{|x|<1}H'(y-X_{s^-})\right)\Pi(dx)ds\right]<\infty \end{align*}
and hence we may write (cf. \eqref{K_decompo})
\be{Kevin31}
f^y(X_t) = f^y(0) + M^y_t + \int_0^t \mathcal{A}f^y(X_{s^-}) \d s \quad \text{for all $t \geq 0$},
\ee
where $M^y$ is a local martingale.

It remains to investigate the sign of $\mathcal{A}f^y$. Plugging in the expression for $H$ from \eqref{def_quad_H} yields
\begin{align*}
\mathcal{A}f^y(l) &= 2\gamma(y-l-b) \left( 1-e^{-\Phi(y-l)} \right)+\sigma^2 \left( 1-e^{-\Phi(y-l)}+(y-l-b)\Phi e^{-\Phi(y-l)} \right) \\
 & \qquad +\int_{(-\infty,0)} \bigg( -2(y-l-b)x+x^2-\frac{2}{\Phi}e^{-\Phi(y-l)}xe^{\Phi x} \\
 & \qquad \qquad \qquad \qquad + 2x\mathbf{1}_{\{ x>-1 \}} \left( y-l-b+\frac{e^{-\Phi(y-l)}}{\Phi}\right) \bigg) \Pi(dx) \\
 & \qquad +\frac{2}{\Phi}e^{-\Phi(y-l)}\left(y-l-b+\frac{1}{\Phi}\right)\int_{(-\infty,0)} \left( e^{\Phi x}-1-x\Phi \mathbf{1}_{\{ x>-1 \}} \right)\Pi(dx).
\end{align*}
Using that $\psi(\Phi)=0$ this simplifies to
\begin{align*}
\mathcal{A}f^y(l) &= \sigma^2+\int_{(-\infty,0)} x^2\Pi(dx)
-\frac{2}{\Phi}e^{-\Phi(y-l)}\left( -\gamma +\Phi\sigma^2 +\int_{(-\infty,0)} \left( xe^{\Phi x}-x \mathbf{1}_{\{ x>-1 \}} \right)\Pi(dx) \right) \\
& \qquad -2(y-l-b)\left( -\gamma+\int_{(-\infty,-1]} x\Pi(dx) \right)
\end{align*}
and further making use of \eqref{eq_psi_der} we can arrive at
\[ \mathcal{A}f^y(l)=\psi''(0)-\frac{2}{\Phi}e^{-\Phi(y-l)}\psi'(\Phi)-2(y-l-b)\psi'(0).\]
Also note that since $\psi'(0)<0$ and $\psi'(\Phi)>0$
\[ \frac{\mathrm{d}}{\mathrm{d}l}\mathcal{A}f^y(l)=2\psi'(0)-\frac{2}{\Phi}\psi'(\Phi)e^{-\Phi(y-l)}< 0 \quad \text{for any $l\leq y$.} \]

Now first consider case (i). In this case plugging in $l=y$ yields
\[ \mathcal{A}f^y(y) = \psi''(0)-\frac{2}{\Phi}\psi'(\Phi)+2b\psi'(0) \geq 0 \]
and hence $\mathcal{A}f^y(l) \geq 0$ for all $l \leq y$. Since by definition of $\kappa_y^0$ the process $(X_{t \wedge \kappa_y^0})_{t \geq 0}$ is bounded above by $y$, the drift term in \eqref{Kevin31} is non-decreasing on $[0,\kappa_y^0]$ and the result follows.

Next consider case (ii). Recalling Lemma \ref{lem_bound} we now have that
\[ \mathcal{A}f^y(y-b)=\psi''(0)-\frac{2e^{-\Phi b}}{\Phi}\psi'(\Phi) > \psi''(0)-\frac{2}{\Phi}\psi'(\Phi) \geq 0. \]
Further $\mathcal{A}f^y(y)<0$. Hence $y^* \in (0,b)$ exists so that $\mathcal{A}f^y(y-y^*)=0$ i.e.
\[ \psi''(0)-\frac{2}{\Phi}e^{-\Phi y^*}\psi'(\Phi)-2(y^*-b)\psi'(0)=0 \]
which confirms that $y^*$ is independent of $y$. Since now $\mathcal{A}f^y(l) \geq 0$ for all $l \leq y-y^*$ the result follows analogue to above, it only remains to show that $y^*<a^*$. The latter follows by observing that $\mathcal{A}f^y(y-a^*)$ is identical to $g_2(a^*)$, where $g_2$ was defined in the proof of Lemma \ref{lem_g}, and $g_2(a^*)>0$ (cf. \eqref{12Apr1}).
\end{proof}

The next four lemmas are dedicated to fully fleshing out the case that
\be{eq_case1}
b > \frac{\psi'(\Phi)/\Phi-\psi''(0)/2}{\psi'(0)}.
\ee

\begin{lemma}\label{lem_stop1} Recall the stopping region $\mathcal{S}$ as defined in Lemma \ref{lem_CS} and $y^*$ as defined in Lemma \ref{lem_Z^y}. Under \eqref{eq_case1} there exists $a_0 \in [y^*,b]$ so that $\mathcal{S} \cap [y^*,\infty)=[a_0,\infty)$. 
\end{lemma}

\begin{proof} First note that since $H$ attains its global minimum in $b$ it immediately follows that $V(b)=H(b)$ i.e. $b \in \mathcal{S} \cap [y^*,\infty)$. Further since $V$ is continuous (cf. Lemma \ref{lem_CS}), $\mathcal{S}=(V-H)^{-1}(\{0\})$ is a closed set. It hence suffices to show that if $\widehat{y} \in \mathcal{S} \cap [y^*,\infty)$ and $y \geq \widehat{y}$, then $y \in \mathcal{S}$. 

For this, using the same notation as in Lemma \ref{lem_Z^y} and recalling Remark \ref{rem_XeqY}, note that
\[ H \left( Y^y_{t\wedge\kappa_y^{\hat y}} \right) = H \left( y-X_{t\wedge\kappa_y^{\widehat{y}}}\right) = Z^y_{t\wedge\kappa_y^{\widehat{y}}} \quad \text{for all $t \geq 0$} \]
and since $\widehat{y} \in \mathcal{S}$
\[ V \left( Y^y_{\kappa_y^{\widehat{y}}} \right) = V(\widehat{y})=H(\widehat{y})=H \left( Y^y_{\kappa_y^{\widehat{y}}} \right) \quad \text{ on $\{\kappa_y^{\widehat{y}}<\infty\}$}. \]
Using $\tau^*$ as defined in Lemma \ref{lem_CS} we can write
\begin{align} 
V(y) = \E \left[ H \left( Y^y_{\tau^*} \right) \right] &= \E\left[ \mathbf{1}_{\{ \tau^*\leq{\kappa_y^{\hat y}} \}} H\left(Y^y_{\tau^*}\right) + \mathbf{1}_{\{ \tau^* > {\kappa_y^{\widehat{y}}}\}} V\left(Y^y_{\kappa_y^{\widehat{y}}}\right)\right] \nonumber\\
&=  \E\left[ \mathbf{1}_{\{ \tau^*\leq {\kappa_y^{\widehat{y}}} \}} H\left(Y^y_{\tau^*}\right) + \mathbf{1}_{\{ \tau^* > {\kappa_y^{\widehat{y}}} \}} H\left(Y^y_{\kappa_y^{\widehat{y}}}\right)\right] \nonumber\\
&=  \E\left[H\left(Y^y_{\tau^* \wedge {\kappa_y^{\widehat{y}}}}\right)\right] = \E\left[ Z^y_{\tau^* \wedge \kappa_y^{\widehat{y}}} \right]. \label{13Apr2}
\end{align}
Denote by $\{T_n\}_{n\in\N}$ the localising sequence of stopping times for the local submartingale from Lemma \ref{lem_Z^y}. Since $\widehat{y} \geq y^*$, it follows from the Optional Sampling Theorem that for any $t>0$, $n\in\N$ and $a>y$
\be{13Apr1} 
\E\left[H\left(Y^y_{\tau \wedge \kappa_y^{\widehat{y}} \wedge T_n\wedge t \wedge \tau_a^y} \right)\right] \geq \E\left[H\left(Y^y_0 \right)\right] = H(y). 
\ee
Clearly for all $a$ large enough, $H$ attains its maximum on $[0,a]$ in $a$, and $\tau_a^y<\infty$ a.s. So the rv in the left hand side of \eqref{13Apr1} is bounded above by $H(Y^y_{\tau_a^y})$ which is integrable on account of Lemma \ref{lem_fa}, and hence by dominated convergence \eqref{13Apr1} remains valid for $t=\infty$ and $n=\infty$. Next we can let $a \to \infty$ and use monotone convergence to arrive at
\[ \E\left[H\left(Y^y_{\tau \wedge \kappa_y^{\widehat{y}}} \right)\right] \geq H(y), \]
so that it follows from \eqref{13Apr2} that $V(y) \geq H(y)$ i.e. $y \in \mathcal{S}$.
\end{proof}

\begin{lemma}\label{lem_stop2} Under \eqref{eq_case1}, $0 \not\in \mathcal{S}$.
\end{lemma}

\begin{proof} Take any $a \in (0,y^*)$ (with $y^*$ as defined in Lemma \ref{lem_Z^y}). We will show that $f_a(0)<H(0)$ so that indeed 
\[ V(0) \leq \E \left[ H \left( Y^y_{\tau^0_a} \right) \right]=f_a(0)<H(0) \]
i.e. $0 \not\in \mathcal{S}$. For this, using the expression from Lemma \ref{lem_fa} we can write
\[ f_a(0)= H(0)-\int_0^a W(z)\left( 2\psi'(0)b+\psi''(0)-2\psi'(0)z-\frac{2\psi'(\Phi)}{\Phi}e^{-\Phi z}\right) \d z+g(a)\frac{W(a)}{W'(a)}. \]
Using the definition of $g_2$ from the proof of Lemma \ref{lem_g} and that $g'=W' g_2$ we can further develop this as
\begin{align}
f_a(0) &= H(0) - \int_0^a W(z)g_2(z) \d z +g(a)\frac{W(a)}{W'(a)} \nonumber\\
&= H(0) + \int_0^a g_2(z)W'(z) \left( \frac{W(a)}{W'(a)}-\frac{W(z)}{W'(z)} \right) \d z +g(0)\frac{W(a)}{W'(a)}. \label{13Apr5} 
\end{align}

From Lemma 2.3 in \cite{Kuznetsov12} we know that 
\[ \frac{W(x)}{W'(x)} = \frac{1}{n(\overline{\epsilon}>x)}, \]
where $n$ is the excursion measure of $X$ from its running supremum and $\overline{\epsilon}$ is the height of a generic excursion. In particular $W/W'$ is hence non-decreasing. Further, since $g_2<0$ on $(0,y^*)$ (cf. the proof of Lemma \ref{lem_g}) and $W'>0$ the integral in \eqref{13Apr5} is non-positive. 

Now, if $X$ has bounded variation then $g(0)<0$ (cf. the proof of Lemma \ref{lem_g}) and hence it indeed follows from \eqref{13Apr5} that $f_a(0)<H(0)$.

In the unbounded variation case we have $g(0)=0$ (cf. the proof of Lemma \ref{lem_g}) and we hence need to show that the integral in \eqref{13Apr5} is strictly negative. If this were not strictly negative then the integral would vanish, which since $W/W'$ is non-decreasing is only possible when $W/W'$ were equal to some constant $C>0$ on $(0,a)$. However this simple ODE only admits $W(x)=Ae^{x/C}$ for some $A>0$ as a positive solution, with $W(0+)=A>0$ and this is impossible since $W(0+)=W(0)=0$ in the unbounded variation case.
\end{proof}

\begin{lemma}\label{lem_stop3} Under \eqref{eq_case1}, $\mathcal{S} \cap (0,a^*)=\emptyset$.
\end{lemma}

\begin{proof} First consider the bounded variation case. For any $y \in (0,a^*)$ it holds that
\[ V(y) \leq \lim_{\eps \downarrow 0} \E \left[ H \left( Y^y_{\tau^y_{y+\eps}} \right) \right] = \lim_{\eps \downarrow 0} f_{y+\eps}(y) = H(y)+\frac{W(0)}{W'(y)}g(y), \]
where the first inequality is just by definition of $V$ and further we used Lemma \ref{lem_fa}. Since $W(0)>0$, $W'(y)>0$ and $g(y) <0$ (cf. Lemma \ref{lem_g}) it follows that $V(y)<H(y)$ i.e. $y \not\in \mathcal{S}$.

Next we consider the unbounded variation case. We first show that $\mathcal{S} \cap (0,y^*)=\emptyset$, with $y^* \in (0,a^*)$ as defined in Lemma \ref{lem_Z^y}. Fix $y \in (0,y^*)$. Pick some $y_0 \in (0,y)$ and $t_0>0$, and define the stopping time
\[ \kappa = \inf \{ t \geq 0 \, | \, Y^y_t \not\in [y_0,y^*] \} \wedge t_0 = \inf \{ t \geq 0 \, | \, y-X_t \not\in [y_0,y^*] \} \wedge t_0. \]
Using that for any $t \leq \kappa$  the process $Y^y$ is strictly positive, we can use the same  notation and decomposition as in Lemma \ref{lem_Z^y} to write 
\[ H(Y^y_t) = f^y(X_t)= H(y)+M^y_t + \int_0^t \mathcal{A}f^y(X_{s^-}) \d s \quad \text{for $t \in [0,\kappa]$,} \]
where $M^y$ is a local martingale whose localising sequence we denote by $\{T_n\}_{n \in \N}$. By construction of $y^*$ we have $\mathcal{A}f^y(X_{t^-})<0$ for $t \leq \kappa$, and hence for any $n \in \N$ it follows from the Optional Sampling Theorem that $\E[H(Y^y_{\kappa\wedge T_n})]<H(y)$. It follows that 
\[ V(y) = \inf_\tau \E \left[ H \left( Y^y_{\tau} \right) \right] \leq \E \left[ H \left( Y^y_{\kappa\wedge T_n} \right) \right] < H(y) \]
i.e. $y \not\in \mathcal{S}$. 

It remains to show that $\mathcal{S} \cap [y^*,a^*)=\emptyset$. If this were not true, then by virtue of Lemma \ref{lem_stop1} it must be the case that $a_0<a^*$ and that $[a_0,a^*] \subseteq \mathcal{S}$. However fix some $y \in (a_0,a^*)$. Then, using Lemma \ref{lem_fa} together with $W(0)=0$, $W'(0)>0$ and $g(y)<0$ (cf. Lemma \ref{lem_g})
\[ f_y(y-)=H(y)+\frac{W(0)}{W'(y)}g(y)=H(y)=f_y(y) \quad \text{and} \quad f_y'(y-)=H'(y)+\frac{W'(0)}{W'(y)}g(y)>H'(y) \]
and hence for all $\eps>0$ small enough 
\[ V(y-\eps) = \inf_\tau \E \left[ H \left( Y^{y-\eps}_\tau \right) \right] \leq \E \left[ H \left( Y^{y-\eps}_{\tau^{y-\eps}_y} \right) \right] =f_y(y-\eps)<H(y-\eps) \]
i.e. $y-\eps \not\in \mathcal{S}$. However this contradicts with $[a_0,a^*] \subseteq \mathcal{S}$.
\end{proof}

\begin{lemma}\label{lem_opt1} Under \eqref{eq_case1}, $\mathcal{S} =[a^*,\infty)$.
\end{lemma}

\begin{proof} From Lemmas \ref{lem_stop1}, \ref{lem_stop2} \& \ref{lem_stop3} it follows that $\mathcal{S}=[a_0,\infty)$ for some $a_0 \geq a^*$, and hence $V(y)=f_{a_0}(y)$ for all $y \geq 0$. It remains to show that $a_0 \leq a^*$. Suppose that $a_0>a^*$.

In the bounded variation case it follows that, using Lemma \ref{lem_fa}
\[ V(a_0-)=f_{a_0}(a_0-)=H(a_0)+g(a_0) \frac{W(0)}{W'(a_0)}>H(a_0), \]
the inequality since $W(0)>0$ and $g(a_0)>0$ (cf. Lemma \ref{lem_g}). But this contradicts $V \leq H$.

In the unbounded variation case we can similarly use Lemma \ref{lem_fa} to see that the left derivative at $a_0$ equals
\[ V_-'(a_0)=H'(a_0)-g(a_0) \frac{W'(0)}{W'(a_0)}<H'(a_0), \]
the inequality again using that $g(a_0)>0$ and that $W'(0)>0$. Combined with $V(a_0)=H(a_0)$ this again contradicts $V \leq H$.
\end{proof}

It remains now to wrap up the other case in Theorem \ref{thm_main} i.e. that
\be{eq_case2}
b \leq \frac{\psi'(\Phi)/\Phi-\psi''(0)/2}{\psi'(0)}.
\ee

\begin{lemma}\label{lem_opt2} Under \eqref{eq_case2}, $\mathcal{S}=[0,\infty)$.
\end{lemma}

\begin{proof} By the same arguments as in Lemma \ref{lem_stop1} we have that $b \in \mathcal{S}$ and that if $y \in \mathcal{S}$ then also $[y,\infty) \subseteq \mathcal{S}$. Hence $\mathcal{S}=[a_0,\infty)$ for some $a_0 \in [0,b]$. Using that now $g>0$ on $(0,\infty)$ (cf. Lemma \ref{lem_g}), the same arguments as in Lemma \ref{lem_opt1} can be used to rule out $a_0>0$. 
\end{proof}

\section{Some examples}\label{sec_exam}

We conclude by discussing some explicit examples. In the below, recall that the function $f_a$ was defined in Lemma \ref{lem_fa}. Further, in case (ii) of Theorem \ref{thm_main} we may write 
\[ V(y)=\inf_{a>0} f_a(y)=f_{a^*}(y) \quad \text{for all $y \geq 0$}, \] 
where $a^*$ is the unique root of the function $g$ (cf. Lemma \ref{lem_g}). Since we have derived expressions for both $f_a$ and $g$ in terms of scale functions, for any spectrally negative L\'evy process with known scale functions we also have explicit expressions for $f_a$ and $g$, and a standard numerical root finding routine can be used to (approximately) compute $a^*$.  

First we consider the bounded variation spectrally negative L\'evy process $X$ given by
\be{eq_proc_ex1} 
X_t= ct-\sum_{k=1}^{N_t} Z_k \quad \text{for all $t \geq 0$,} 
\ee
where $c>0$, $(N_t)_{t \geq 0}$ is a Poisson process with intensity $\mu>0$ and the $Z_k$'s are iid with a common $\text{Exp}(\eta)$ distribution for some $\eta>0$. We assume $\mu>c \eta$ so that the conditions in \eqref{eq_ult_sup} hold.

It is well known (and easily verified) that
\[ \psi(z)=cz-\mu+\frac{\mu \eta}{z+\eta} \quad \text{with domain $(-\eta,\infty)$.} \]
Note that both assumptions (A1) and (A2) are satisfied, and that $\Phi=\mu/c-\eta$. Further it is well known (and again easily verified) that
\[ W(x) = \frac{\mu}{c(\mu -\eta c)}e^{\left(\mu/c-\eta\right)x}-\frac{\eta}{\mu-\eta c} \quad \text{for $x\geq 0$}. \]

For the particular choices $c=0.5$, $\mu=1$, $\eta=1$ and $b=5$ we are in case (ii) of Theorem \ref{thm_main} and we can compute $a^* \approx 3.995<b=5$. 

Figure \ref{fig:fayCPP} illustrates this situation, where the functions $H$ and $f_a$ for several values of $a$ are plotted, including $V=f_{a^*}$. Of course, $V=f_{a^*}$ is dominated by every other $f_a$. Note that each $f_a$ is constant for $y \in [0,a)$ due to the fact that $Y^y$ moves upwards by jumps only and that due to the lack of memory property of the Exponentially distributed jumps the overshoot over the level $a$ does not depend on the position just prior to the jump causing the overshoot. Further it is worth noting that $V=f_{a^*}$ is the only choice of $a \geq 0$ that results in a continuous function, for $a \not= a^*$ the function $f_a$ experiences a discontinuity in $y=a$. This phenomenon is well known in optimal stopping (for processes of bounded variation) and is usually referred to as `continuous fit' or `continuous pasting' (see e.g. \cite{Alili05}).

\begin{figure}
\centering
\includegraphics[trim={0 0 3cm 1cm},clip]{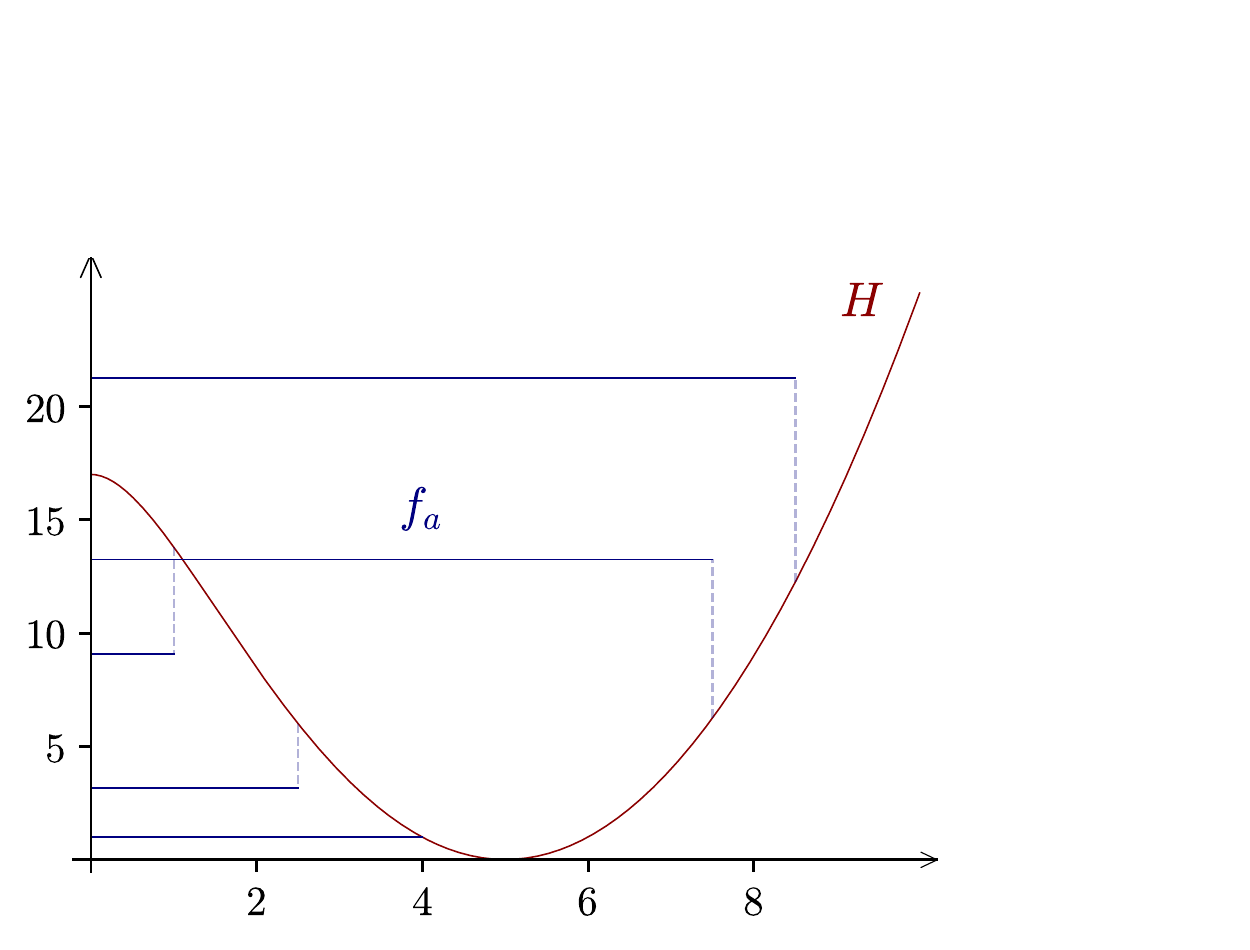}
\caption{A plot of $H$ and $f_a$ for several values of $a$ for the bounded variation process \eqref{eq_proc_ex1} with $c=0.5$, $\mu=1$, $\eta=1$ and $b=5$}
\label{fig:fayCPP}
\end{figure} 

Next we add a Brownian motion $B$ to the above example to create a process of unbounded variation:
\be{eq_proc_ex2} 
X_t= \sigma B_t + ct-\sum_{k=1}^{N_t} Z_k \quad \text{for all $t \geq 0$.} 
\ee
It is again easily verified that (see also e.g. \cite{Kuznetsov12})
\[ \psi(z)=\frac{\sigma^2}{2} z^2 + cz-\frac{\mu z}{z+\eta} \quad \text{with domain $(-\eta,\infty)$} \]
and denoting the three distinct roots of $\psi$ by $z_i$
\[ W(x) = \frac{e^{z_1 x}}{\psi'(z_1)}+\frac{e^{z_2 x}}{\psi'(z_2)}+\frac{e^{z_3 x}}{\psi'(z_3)} \quad \text{for $x \geq 0$.} \]
Choosing $\sigma=1$, $c=0.5$, $\mu=1$, $\eta=1$ and $b=5$ we are again in case (ii) of Theorem \ref{thm_main} and we can compute $a^* \approx 4.38<b=5$.

Figure \ref{fig:fayBM} again shows plots of $H$ and $f_a$ for several values of $a$. Noteworthy is that now all the $f_a$'s are continuous, and that $V=f_{a^*}$ distinguishes itself from $f_a$ for other values of $a$ via a $C^1$ fit to $H$ at $y=a$ rather than only a continuous fit. This is known as `smooth fit' or `smooth pasting' (see again \cite{Alili05} e.g.).

\begin{remark} It is worth noting that in Figures \ref{fig:fayCPP} \& \ref{fig:fayBM} the functions $f_a$ seem to have a vanishing derivative in $y=0$. Indeed a bit of algebra with the expressions from Lemma \ref{lem_fa} reveals that in general $f_a'(0+)=0$. This property is usually referred to as `normal reflection' or `instantaneous reflection' (see e.g. p. 264 in \cite{Peskir06} or equation (3.32) in \cite{DuToit09}) and naturally arises as follows. Using the strong Markov property it can be checked for any $t \geq 0$ that
\[ \E \left[ \left. H \left( Y^y_{\tau^y_a} \right) \, \right| \, \mathcal{F}_t \right] = f_a \left( Y^y_{t \wedge \tau^y_a} \right) =: Z_t \]
i.e. that $Z$ is martingale. (Note that the Snell envelope for our problem is contained in this family, for $a=y^*$). If the normal reflection property would not hold, then the (local) time $Y^y$ accumulates at level $0$ would destroy this martingale property. In a more pure free boundary approach to problems of this type normal reflection is typically included as one of the conditions imposed on the solution to the free boundary problem. However in the approach we have chosen in this paper, where in particular the proofs only consider processes that are stopped before $Y^y$ hits the level $0$, normal reflection appears here as a property of the solution rather than as a condition imposed on the solution.
\end{remark}

\begin{figure}[!hb]
\centering
\includegraphics[trim={0 0 1cm 1cm},clip]{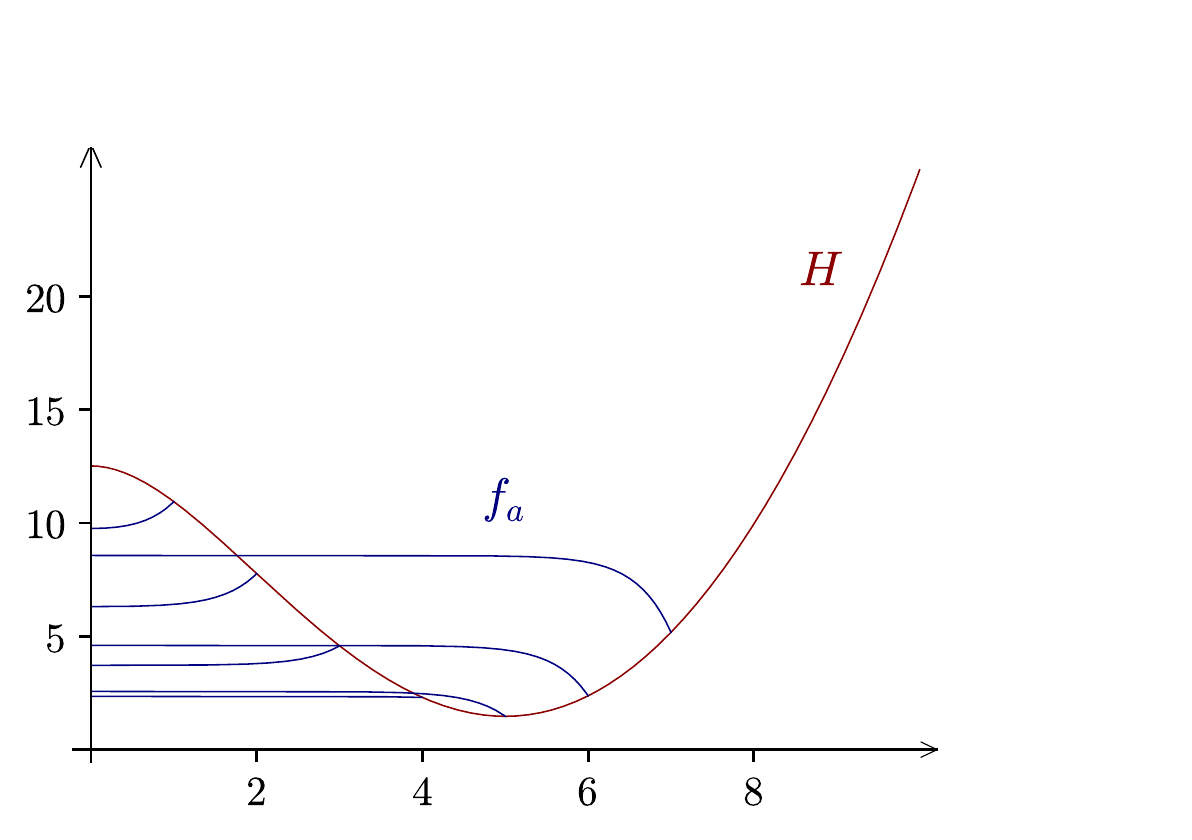}
\includegraphics[trim={0 0 1cm 1cm},clip]{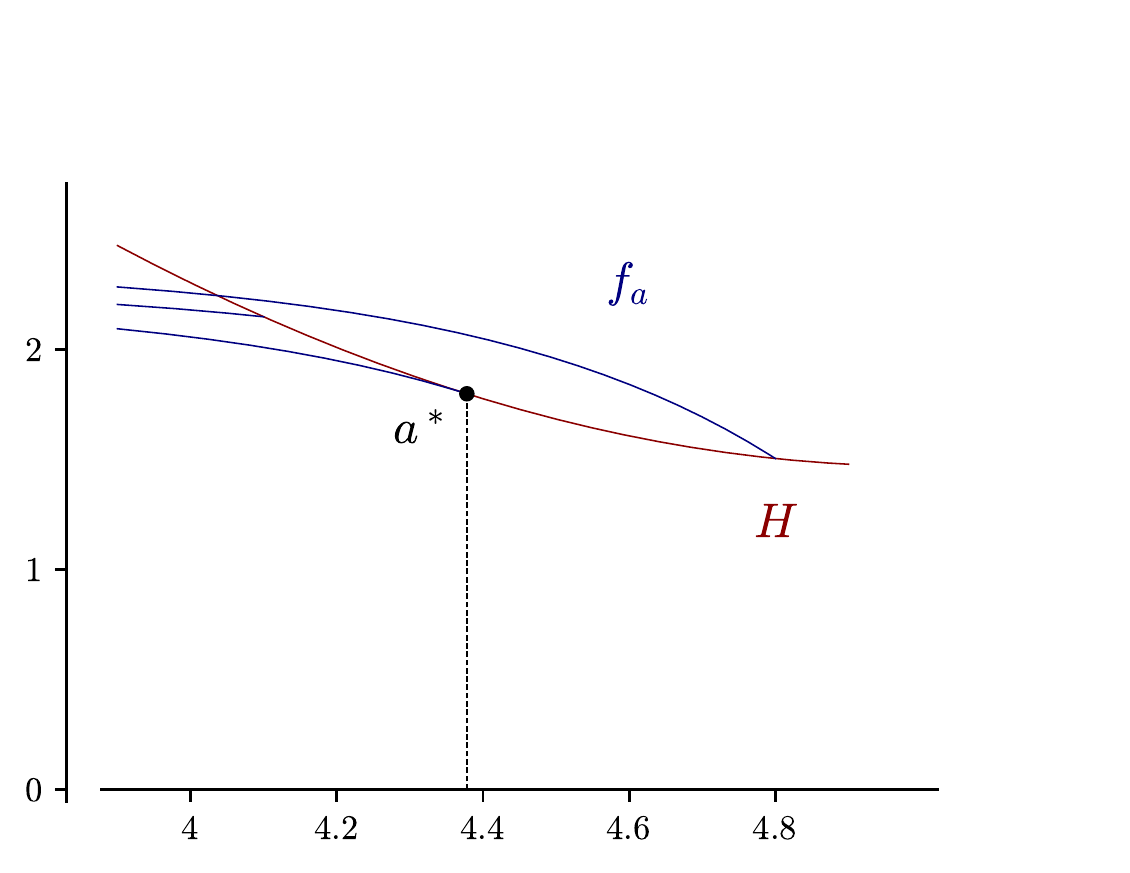}
\caption{A plot of $H$ and $f_a$ for several values of $a$ for the unbounded variation process \eqref{eq_proc_ex2} with $\sigma=1$, $c=0.5$, $\mu=1$, $\eta=1$ and $b=5$}
\label{fig:fayBM}
\end{figure}

\end{document}